\documentclass{amsart}
\usepackage{titlecaps}

\usepackage{ mathrsfs }
\usepackage{ amsfonts }
\usepackage{ dsfont }
\usepackage{ amsthm }
\usepackage{ amsmath }
\usepackage{ amssymb }
\usepackage{ verbatim }
\usepackage{ tikz-cd }
\usepackage{ mathtools }
\usepackage{ color }
\usepackage{hyperref}
\usepackage{url,amssymb,enumerate,colonequals}
\usepackage{xcolor}

\usepackage[symbol]{footmisc}

\newtheorem{lem}{Lemma}[section]
\newtheorem{thm}{Theorem}[section]

\newcommand{\F}{\mathbb{F}}
 
\newcommand{\ZZ}{\mathbb{Z}}

\newcommand{\FF}{\mathbb{F}}

\newcommand{\QQ}{\mathbb{Q}}

\DeclareMathOperator{\ord}{ord}
\DeclareMathOperator{\Primes}{Primes}
\DeclareMathOperator{\redl}{red_{\ell}}

\DeclareMathOperator{\redt}{red_{2}}

\title{Torsion primes for elliptic curves over degree 8 number fields}

\author{Maleeha Khawaja}
\address{School of Mathematics and Statistics, Hicks Building, University of Sheffield, Sheffield S3 7RH, United Kingdom}
\email{mkhawaja2@sheffield.ac.uk}

\date{\today}

\keywords{Modular curves, elliptic curves, rational points, abelian varieties\\
Data sharing is not applicable to this article as no datasets were generated or analysed during the current study.}

\makeatletter
\@namedef{subjclassname@2020}{%
  \textup{2020} Mathematics Subject Classification}
\makeatother


\begin{document}

\begin{abstract} 
Let $d\geq 1$ be an integer and let $p$ be a rational prime. 
    Recall that $p$ is a torsion prime of degree $d$ if 
    there exists an elliptic curve $E$ over a degree $d$ number field $K$ such that $E$ has a $K$-rational point of order $p$. 
    Derickx, Kamienny, Stein and Stoll \cite{derickx2021torsion} have computed the torsion primes of degrees 4, 5, 6 and 7. We verify that the techniques used in \cite{derickx2021torsion} can be extended to determine the torsion primes of degree 8. 

\end{abstract}

\maketitle

\section{Introduction}

Let $d\geq 1$ be an integer and let $p$ be a rational prime. 
Recall that $p$ is a \textbf{torsion prime of degree $d$} if there is a number field $K$ of degree $d$ and an elliptic curve $E$ over $K$ with a $K$-rational point of order $p$.
Let $S(d)$ denote the set of torsion primes of degree $d$. Let Primes$(x)$ denote the set of primes less or equal than $x$. 
Mazur \cite{Mazur77, Mazur78} was the first to completely determine the set $S(d)$ for any integer $d$. He found that
$S(1)=\Primes(7)$.
Kamienny \cite{Kamienny} and Parent \cite{Pierre000, Parent003} determined the torsion primes of degrees 2 and 3, respectively, finding that $S(2)=S(3)=\Primes(13)$.
Building on the techniques used in these works, Derickx, Kamienny, Stein and Stoll \cite{derickx2021torsion} proved the following result.
\begin{thm}[Derickx, Kamienny, Stein and Stoll]
 \label{DKSS}
For an integer $d\geq 1$, let $S(d)$ be the set of torsion 
primes of degree $d$. Then,
\begin{multline*}
    S(4)=\Primes(17),\; S(5)=\Primes(19),\\
    S(6)=\Primes(19)\cup\{37\}, \text{ and } S(7)=\Primes(23).
\end{multline*}

\end{thm}

We use the techniques and computations of the aforementioned paper to determine the set of torsion primes of degree 8.

\begin{thm}
    \label{thm:S8}
For an integer $d\geq 1$, let $S(d)$ be the set defined above. Then,
\[
S(8)=\Primes(23).
\]
\end{thm}

Although the study of low degree points on curves has received much attention in and of itself, the precise determination of the set of torsion primes of low degree has also had several applications to the explicit resolution of Diophantine equations; see e.g. \cite{Anni16}, \cite{Freitas14}, \cite{khawaja2022fermats}. 
We expect Theorem \ref{thm:S8} to have Diophantine 
applications in a similar manner.

\medskip

The purpose of this note is to provide a proof of Theorem \ref{thm:S8}. We stress that none of the ideas or techniques used in this note are due to us and we are merely verifying that the techniques used to prove Theorem \ref{DKSS} can be extended to prove Theorem \ref{thm:S8}.

\medskip

All computations were performed in \texttt{Magma} \cite{Magma} using Stoll's code \cite{StollCode}. 
All supporting computations can be found at
\begin{center}
    \url{https://github.com/MaleehaKhawaja/deg8torsionprimes}.
\end{center}

\medskip

After completing our computations we learnt that Maarten Derickx and Michael Stoll have independently determined the torsion primes of degree 8 in unpublished work, as well as computing smaller bounds for the sets $S(9)$ and $S(10)$.

\medskip

The following two results form the theoretical basis of the proof. For the benefit of the reader, we include the proofs of these results here.

\medskip

Let $d\geq 1$ be an integer over $\QQ$. 
Let $X$ be a curve over $\QQ$, and let $X^{(d)}$ denote the d$^{th}$ symmetric 
power of $X$. 
Recall that points in $X^{(d)}(\QQ)$ 
correspond to effective degree $d$ divisors 
on $X$. 
Let $N\geq 1$ be an integer and suppose 
$X=X_{1}(N)$. 
Then $C_{1}(N)$ denotes the set of cusps on $X_{1}(N)$.

\begin{lem}[Derickx, Kamienny, Stein and Stoll]\label{lem1.5}
Let $d\geq 1$ be an integer, and let $p$ be a prime. Let $\alpha$ be defined by the following composition of maps
\begin{equation}
    \alpha: C_{1}(p)(\QQ)^{d}\rightarrow X_{1}(p)(\QQ)^{d}\rightarrow X_{1}(p)^{(d)}(\QQ).
\end{equation}
If $\alpha$ is surjective then $p\notin S(d)$.
\end{lem}

\begin{proof} 
This is \cite[Lemma 1.5]{derickx2021torsion}.
Suppose, for a contradiction that $p\in S(d)$.
Thus there is a non--cuspidal $K$--rational point on $X_{1}(p)$.
Thus there is a pair $(E,P)$ where $E$ is an elliptic curve over $K$ and $P$ is a point of order $p$ on $K$. Taking the sum of the Galois conjugates of $P$ gives a rational effective degree $d$ divisor on $X_{1}(p)$. 
This divisor isn't the sum of rational cusps - this contradicts the surjectivity of $\alpha$.
\end{proof}
For a prime $\ell$ distinct from $p$, let $\redl$ denote the usual reduction map:
\begin{equation*}
    \redl:X_{1}(p)^{(d)}(\QQ)\rightarrow X_{1}(p)^{(d)}(\mathds{F}_{\ell}).
\end{equation*}

\begin{lem}[Derickx, Kamienny, Stein and Stoll]\label{MainStrategy}
Let $\ell$ be a prime distinct from $p$. Let $\bar{x}\in X_{1}(p)^{(d)}(\mathds{F}_{\ell})$. Suppose that the following two assumptions are satisfied.
\begin{enumerate}
    \item[(a)] If $\bar{x}$ is the sum of images of rational cusps under $\redl$ then the residue class of $\bar{x}$ contains at most one rational point.
    \item[(b)] If $\bar{x}$ is not the sum of images of rational cusps under $\redl$ then the residue class of $\bar{x}$ doesn't contain a rational point.
\end{enumerate}
Then $p\not\in S(d)$.
\end{lem}

\begin{proof}
This is \cite[Lemma 1.7]{derickx2021torsion}. 
We want to show that the map $\alpha$, as defined in Lemma \ref{lem1.5}, is surjective. Let $x\in X_{1}(p)^{(d)}(\QQ)$ be a rational point in the residue class of $\bar{x}$. By taking the contrapositive of assumption $(b)$, $\bar{x}$ is the sum of images of rational cusps under $\redl$. That is 
\begin{equation*}
    \bar{x}=\bar{x}_{1}+...+\bar{x}_{d}
\end{equation*}
where $\bar{x}_{i}=\redl({x}_{i})$ for rational cusps $x_{i}\in X_{1}(p)^{(d)}(\QQ)$. 
Let 
\begin{equation*}
    x'\coloneqq x_{1}+...+x_{d}\in X_{1}(p)^{(d)}(\QQ). 
\end{equation*}

Since $x'$ is a rational point in the residue class of $\bar{x}$, 
it immediately follows from assumption (a) that $x=x'$. 
It follows that $x\in C_{1}(p)(\QQ)^{d}$ since $x=x'$ is the sum of rational cusps. 
By Lemma \ref{lem1.5} we have $p\not\in S(d)$.

\end{proof}

Since $\Primes(23)\subseteq S(8)$ (see \cite[Proposition 1.3]{derickx2021torsion}), in order to prove Theorem \ref{thm:S8}, it remains to prove that the reverse inclusion holds. The smallest general bound for torsion primes of degree $d$ is due to Oesterl\'{e}: 
\begin{equation*}
    S(d)\subset \Primes((3^{d/2}+1)^2);
\end{equation*}
see \cite[Section 6]{derickx2021torsion} for a proof.
In particular we need to verify that both assumptions of Lemma \ref{MainStrategy} hold for $29\leq p < 6724$. We say $p$ is a \textbf{rank zero prime} if the Jacobian 
$J_{1}(p)$ of the modular curve $X_{1}(p)$ has rank $0$ 
over $\QQ$. 
In Section \ref{sec:rankzero}, we verify that 
Lemma \ref{MainStrategy} holds for all rank zero primes. 
In Sections \ref{sec:assumptiona} and 
\ref{sec:assumptionb}, we verify that assumptions (a) and (b) of Lemma \ref{MainStrategy} hold for all remaining primes. 

\section*{acknowledgements}
The author is sincerely grateful to Frazer Jarvis and Michael Stoll for helpful correspondence, and would like 
to thank the anonymous referees for their invaluable 
feedback
and careful reading of a previous version of the paper. 
The author is supported by an EPSRC studentship from the University of Sheffield
(EP/T517835/1). 

\section{Rank zero primes}
\label{sec:rankzero}
We begin our verification at the primes $p$ for which the Jacobian $J_{1}(p)$ of $X_{1}(p)$ has rank 0 over $\QQ$. If $p$ is such a prime then we refer to $p$ as a \textbf{rank zero prime}. 
There are finitely many rank zero primes, 
and moreover $p$ is a rank zero prime if and only if
\begin{equation*}
    p\leq 31 \qquad \text{or}\qquad p\in\{41, 47, 59, 71\};
\end{equation*}
see \cite[Proposition 6.2.1]{Conrad03}.

\medskip

Let $X$ be a curve defined over $\QQ$.
Let $d\geq 1$ be an integer. 
Fix $x_{0}\in X^{(d)}(\QQ)$. 
Let $J$ be the Jacobian of $X$.
Recall that the Abel--Jacobi map $\iota$ is given 
by 
\[
X^{(d)}\rightarrow J,\quad x\mapsto [x-x_{0}].
\]
Recall that the $\QQ$-\textbf{gonality} of $X$ is 
the minimum degree of a map from $X$ to $\mathds{P}^{1}$ defined over $\QQ$. 
We break the proof of \cite[Corollary 3.5]{derickx2021torsion} into smaller parts.

\begin{lem}
    \label{lem:goninj}
    Let $d\geq 1$ be an integer. 
    Suppose $p$ is a prime such the $\QQ$-gonality of $X_{1}(p)$ is strictly greater than $d$. Then the map
    \begin{equation*}
        \iota: X_{1}(p)^{(d)}(\QQ)\rightarrow J_{1}(p)(\QQ)
    \end{equation*}
    is injective.
\end{lem}

\begin{proof}
    Suppose there exist $x_{1}, x_{2}\in X_{1}(p)^{(d)}(\QQ)$ such that $\iota(x_{1})=\iota(x_{2})$. 
    Then
    \begin{equation*}
        x_{1}-x_{2}=D_{1}-D_{2}+(f),
        \qquad D_{1}, D_{2}\in J_{1}(p)(\QQ),\;
        f\in L(D_{1}-D_{2}).
    \end{equation*}
    In particular, the degree of $f\in\QQ(X_{1}(p))$ is less than or equal to $d$. 
    This contradicts the assumption on the $\QQ$-gonality of $X_{1}(p)$. 
\end{proof}

 \begin{lem}[Derickx, Kamienny, Stein and Stoll]
        \label{lem:rankzero8a}
 Let $d\geq 1$ be an integer. Suppose $p\geq 3$ is a rank zero prime such that the $\QQ$-gonality of $X_{1}(p)$ is strictly greater than $d$. 
 Then assumption (a) of Lemma \ref{MainStrategy} is satisfied for $p$ with $\ell=2$.
 \end{lem}

 \begin{proof}
     This is \cite[Corollary 3.5]{derickx2021torsion}.
     We recall that we need to show that the reduction map 
     $\redt: X_{1}(p)^{(d)}(\QQ)\rightarrow X_{1}(p)^{(d)}(\mathds{F}_{2})$ is injective.
     It follows from Lemma \ref{lem:goninj} and \cite[Proposition 3.4]{derickx2021torsion} that the map $\redt \circ\; \iota=\iota \circ \redt$ is injective. Thus $\redt$ is injective.
 \end{proof}

Let $p\in\{29, 31, 41, 47, 59, 71\}$. 
Then it follows from work of Derickx and van Hoeij \cite{Derickx14} that the $\QQ$-gonality of $X_{1}(p)$ is strictly greater than $8$. Thus assumption (a) of Lemma \ref{MainStrategy} is satisfied with $\ell=2$ by Lemma \ref{lem:rankzero8a}.
We now verify that assumption (b) holds for these primes. One way to do this is to show that every $\bar{x}\in X_{1}(p)^{(d)}(\mathds{F}_{\ell})$ is the sum of images of rational cusps.

\begin{lem}
    \label{lem:29notinS8}
    Let $p=29$, $31$ or $41$.
    Then $p\not\in S(8)$.
\end{lem}

\begin{proof}
    From the remarks above, it remains to demonstrate that assumption (b) of Lemma \ref{MainStrategy} is satisfied. We follow the proof of \cite[Lemma 3.7]{derickx2021torsion}. 
    
    Write $X=X_{1}(p)$ and $J=J_{1}(p)$.
    By \cite[Corollary 3.3]{derickx2021torsion}, $J(\QQ)$ is generated by the differences of rational cusps. Thus, if there is a rational point in the residue class of $\bar{x}\in X^{(8)}(\mathds{F}_{2})$ then $\bar{x}$ maps into the subgroup of $J(\mathds{F}_{2})$ that is generated by the differences of the images of the rational cusps.
    We use Stoll's code \cite{StollCode} to verify that, under the hypothesis of assumption (b), $\bar{x}$ doesn't map into this subgroup. 
    The supporting computations can be found in 
    the script \texttt{X129\_31\_41.m}.
\end{proof}


To verify that assumption (b) holds for the primes $47$, $59$, and $71$, we shall need the following lemma.

\begin{lem}[Derickx]
    \label{lem:principal}
Let $d\geq 1$ be an integer, and let $p\geq 3$ be a prime. Suppose $t\in\mathds{T}$ kills the rational points on the Jacobian of $X_{1}(p)$ i.e. 
\begin{equation*}
    t(J_{1}(p)(\QQ))=\{0\},
\end{equation*}
where $\mathds{T}$ denotes the endomorphism ring of $J_{1}(p)$.
Consider two points $\bar{x}_{0},\bar{x}\in X_{1}(p)^{(d)}(\mathds{F}_{2})$ where $\bar{x}_{0}$ is a sum of images of rational cusps. If the divisor $t(\bar{x}-\bar{x}_{0})$ is not principal then there is no rational point $x\in X_{1}(p)^{(d)}(\QQ)$ in the residue class of $\bar{x}$. 
\end{lem}

\begin{proof}
This is \cite[Lemma 8.6]{derickx2021torsion}. 
Let $X=X_1(p)$ and $J=J_{1}(p)$. Suppose there is a rational point $x\in X^{(d)}(\QQ)$ in the residue class of $\bar{x}$. 
Let $x_{0}\in X^{(d)}(\QQ)$ be a point in the residue class of $\bar{x}_{0}$ that is the sum of rational cusps. By assumption, $t$ sends points in $J(\QQ)$ to zero. 
Thus $t(x-x_{0})=0$ i.e. $t(x-x_{0})$ is principal. 
In particular, it follows that $t(\bar{x}-\bar{x}_{0})$ is principal.
\end{proof}

\begin{lem}
    \label{lem:rankzero}
    Let $p=47$, $59$ or $71$. 
    Then $p\not\in S(8)$.
\end{lem}

\begin{proof}
    Let $p=47$, $59$ or $71$.
    It remains to show that assumption (b) of Lemma \ref{MainStrategy} is satisfied. 
    Let $X=X_1(p)$ and $J=J_1(p)$. 
    Let $x$ be a degree $8$ point on $X=X_{1}(p)$ and write $\bar{x}$ for the corresponding divisor on $X_{\FF_{2}}$. 

    First consider $p=47$. 
    Using Stoll's code, we checked that there are no elliptic curves over $\FF_{2^{d}}$ 
    with a point of order $p=47$ for $1\leq d\leq 7$. 
    Thus $\bar{x}$ is the sum of eight rational cusps.
    Recall that assumption (a) of Lemma \ref{MainStrategy} with $\ell=2$ follows 
    directly from Lemma \ref{lem:rankzero8a}. 
    Thus $x$ is the sum of eight rational cusps.
    This contradicts the assumption that $x$ is a degree $8$ point on $X=X_{1}(p)$. 

    Now suppose $p=59$ or $71$. Write $T_n$ for the $n$-th Hecke operator, and $\langle a \rangle$ for the diamond operator. Let $t=(\langle 3 \rangle-1)(T_3-3\langle 3 \rangle -1) \in \mathbb{T}$. 
    We checked using \texttt{Magma} that the positive rank simple factors of $J_{1}(p)$ already occur in $J_{0}(p)$.
    Using Stoll's code, we checked that there are no elliptic curves over $\FF_{2^{d}}$ with a point of order $p$ for $1\leq d\leq 6$. 
    As before $\bar{x}$ can not be the sum of eight rational cusps. The only remaining possibility is that $\bar{x}=\bar{D}+\bar{y}$ where $\bar{D}$ is a degree $7$ place on $X_{\F_2}$, and $\bar{y}$ is the reduction of a rational cusp. We note that $t(\bar{y})$ is principal, and $t(\bar{x})$ is principal. Hence $t(\bar{D})$ must be principal. Using Stoll's code we checked that for all degree $7$ places $\bar{D}$ of $X_{\F_2}$, the divisor $t(\bar{D})$ is not principal, giving a contradiction in this case. 
All supporting computations can be found in the script \texttt{rankzeroprimes.m}.
\end{proof}

\section{Verifying assumption (a)}\label{sec:assumptiona}

Let $\ell$ and $p$ be distinct primes. 
Suppose $\bar{x}\in X_{1}(p)^{(d)}(\mathds{F}_{\ell})$. 
We recall assumption (a) of Lemma \ref{MainStrategy}: if $\bar{x}$ is the sum of images of rational cusps under $\redl$ then the residue class of $\bar{x}$ contains at most one rational point. Before stating the main strategy used to verify that this assumption holds, we state an important result that is used in the proof. 


\begin{thm}[Derickx, Kamienny, Stein and Stoll]\label{thm:criterion}
Let $d\geq 1$ be an integer and let $\ell$ and $p$ be distinct primes. Let $t: J_{1}(p)\rightarrow\mathcal{A}$ be a morphism of abelian schemes over $\ZZ_{(l)}$ such that:
\begin{enumerate}
    \item[i)] $t(J_{1}(p)(\QQ))$ is finite;
    \item[ii)] $\ell>2$ or $\#t(J_{1}(p)(\QQ))$ is odd;
    \item[iii)] $t\circ \iota$ is a formal immersion at all $\bar{x}\in X_{1}(p)^{(d)}(\mathds{F}_{\ell})$ that are sums of images of rational cusps on $X_{1}(p)$.
\end{enumerate}
Then assumption (a) of Lemma \ref{MainStrategy} holds.
\end{thm}

\begin{proof}
    This is \cite[Corollary 4.3]{derickx2021torsion}.
    Suppose $x,x^\prime\in X_{1}(p)^{(d)}(\QQ)$ are in the residue class of $\bar{x}$, where $\bar{x}$ is the sum of images of rational cusps. We want to show that $x=x^\prime$. 
    Let $y=(t\circ\iota)(x)$ and $y^\prime=(t\circ\iota)(x^\prime)$. We note that $y$ and $y^\prime$ lie in the same residue class, since $x$ and $x^\prime$ do. 
    It then follows from i) and ii) that $y=y^\prime$. Under the hypothesis of assumption iii), a 
    result of Parent \cite[Lemma 4.13]{Parent99} 
    asserts that the map
    \begin{equation*}
        t\circ\iota: \redl^{-1}(\bar{x})\rightarrow \redl^{-1}((t\circ \iota)(\bar{x}))
    \end{equation*}
    is an injection. 
    Thus $x=x^\prime$.
\end{proof}

We work with $\ell=2$, and refer the reader to 
\cite[Section 5]{derickx2021torsion} for the construction of an appropriate operator $t\in \mathbb{T}$. 
Indeed Derickx, Kamienny, Stein and Stoll \cite[Corollary 5.2]{derickx2021torsion} prove that such a $t$ satisfies assumptions i) and ii) of Theorem \ref{thm:criterion} with $\ell=2$, and one can apply Kamienny's criterion \cite[Proposition 5.3]{derickx2021torsion} to verify assumption iii).
We used Stoll's code to test whether Kamienny's criterion holds, and found that it holds for 
\begin{equation}
    \label{eq:kam}
        137< p < 6724,\quad p\neq 149, 157, 163, 193, 431.
\end{equation}
This part of the verification required several parallel computations. 
The supporting code can be found in script \texttt{assumptiona.m}. 

\medskip

For the primes excluded by \eqref{eq:kam} 
we used Stoll's code to verify that assumption (a) holds by replacing 
$X_{1}(p)$ with an intermediate curve between $X_{1}(p)$ and $X_{0}(p)$, see \cite[Corollary 4.4]{derickx2021torsion}; 
this criterion is similar to Theorem \ref{thm:criterion}.
The supporting code can be found in the script \texttt{remainingprimes.m}.

\section{Verifying assumption (b)}\label{sec:assumptionb}
Let $p$ be a prime, and let $\ell$ be a prime distinct from $p$. Suppose $\bar{x}\in X_{1}(p)^{(d)}(\mathds{F}_{\ell})$. We recall assumption (b) of Lemma \ref{MainStrategy}: 
if $\bar{x}$ isn't the sum of images of rational cusps under $\redl$
then the residue class of $\bar{x}$ doesn't contain a rational point.
As remarked in \cite[pg. 272]{derickx2021torsion}, to verify this assumption, it suffices to show that
\begin{enumerate}
    \item[i)] there is no elliptic curve $E$ over $\mathds{F}_{\ell^{d^\prime}}$, for all $d^\prime\leq d$, such that $p\mid \#E(\mathds{F}_{\ell^{d^\prime}})$;
    \item[ii)] $p\nmid \ell^{d^\prime}\pm 1$ for all $d^\prime\leq d$.
\end{enumerate}

It immediately follows 
from a result of Waterhouse \cite[Theorem 4.1]{Waterhouse} that 
\[
\#E(\mathds{F}_{2^d})\in \{r\in [\lceil(2^{d/2}-1)^2\rceil, \lfloor(2^{d/2}+1)^2\rfloor]\; : \; r \text{ is even}\} \cup \{r_{d}\}
\]
where 
\[
    r_{d}=
    \begin{cases}
        2^d+m2^{d/2}+1,\; m\in\{-2,-1,0,1,2\} & \text{if $d$ is even}\\
        2^d+m2^{(d+1)/2}+1,\; m\in\{-1,0,1\} & \text{if $d$ is odd}.
    \end{cases}
\]
Thus it remains to 
show that assumption (b) holds for
\begin{align*}
    p\in
    \{\,
    & 29,31,37,41,43,47,59,61,67,71,\\
    & 73,113,127,131,137,139,241,257\}.
\end{align*}

By Lemma \ref{lem:29notinS8} and Lemma \ref{lem:rankzero}, we have $29, 31, 41, 47, 59, 71\not\in S(8)$. Thus, it remains to verify that assumption (b) holds for 
\begin{equation*}
    p\in \{37,43,61,67,73,113,127,131,137,139,241,257\}.
\end{equation*}

We fix some notation for the remainder 
of the paper.
Let $X=X_1(p)$ and $J=J_1(p)$. Write $T_n$ for the 
$n$-th Hecke operator, and $\langle a \rangle$ for the diamond operator.
Let $t=(\langle 3 \rangle-1)(T_3-3\langle 3 \rangle -1) \in \mathbb{T}$. 
As shown in \cite[pg. 303]{derickx2021torsion}, the operator $T_3-3\langle 3 \rangle -1$ kills 
rational torsion.
For the relevant primes $p$, we verify that 
the operator $\langle 3\rangle -1 $ maps $J$ into an abelian subvariety 
of rank zero. 
Thus for such $p$ the operator $t$ kills $J(\QQ)$.

\begin{lem}
    \label{lem:smallprimes}
        Let $p\in \{43, 61, 67, 73\}$. Then $p\not\in S(8)$. 
\end{lem}

\begin{proof}
Suppose $x$ is a degree $8$ point on $X=X_{1}(p)$ and denote by $\bar{x}$ the corresponding divisor on $X_{\FF_{2}}$. 
If $p=61, 67$ or $73$ then the operator $\langle 3\rangle -1 $ maps $J=J_{1}(p)$ into an abelian subvariety 
of rank zero by the proof of \cite[Lemma 8.7]{derickx2021torsion}. 
If $p=43$, we checked using \texttt{Magma} that that the positive-rank simple factors of $J_1(43)$ all occur in $J_0(43)$.

    First suppose $p=43$, $61$ or $67$. 
    Using Stoll's code, we checked that there are no elliptic curves over $\FF_{2^{d}}$ with a point of order $p$ for $1\leq d\leq 6$. Thus all places on $X_{\FF_{2}}$ of those degrees $d$ must be cuspidal. 
    We proved that assumption (a) of Lemma \ref{MainStrategy} holds in Section \ref{sec:assumptiona}. 
    Then $\bar{x}=\bar{D}+\bar{y}$ where $\bar{D}$ is a degree $7$ place on $X_{\FF_{2}}$, and $\bar{y}$ is the reduction of a rational cusp. 
    Since $t(\bar{x})$ and $t(\bar{y})$ are principal, it must be that $t(\bar{D})$ is principal. Using Stoll's code, we checked that $t(\bar{D})$ is not principal for all degree $7$ places $\bar{D}$ on $X_{\FF_{2}}$. This gives a contradiction. 

Now suppose $p=73$. 
Using Stoll's code, we checked that there are no elliptic curves over $\FF_{2^{d}}$ with a point of order $p$ for $1\leq d\leq 5$. 
In each possible case, the support of $\bar{x}$ must contain a degree $d$ place $\bar{D}$ such that $t(\bar{D})$ is principal where $d=6, 7$ or $8$. 
Again, we checked that $t(\bar{D})$ is not principal for all degree $d$ places $\bar{D}$ on $X_{\FF_{2}}$.
The supporting computations can be found in the script \texttt{smallprimes.m}. 
\end{proof}

\begin{lem}\label{lem:largeprimes}
    Let $p\in\{113, 127, 131, 137, 139, 241, 257\}$. Then $p\not\in S(8)$.
\end{lem}

\begin{proof}
    Let $p$ be a prime as above. 
    We follow the proof of \cite[Corollary 7.2]{derickx2021torsion}.
    We checked using \texttt{Magma} that the positive rank simple factors of $J_{1}(p)$ already occur in $J_{0}(p)$.
    In order to apply \cite[Proposition 7.1]{derickx2021torsion}, it suffices to find a primitive root $a$ 
    modulo $p$.
    For $p\neq 131, 241$ we choose $a=3$; for 
    $p=131$ we choose $a=2$; for $p=241$ we choose $a=7$. 
    Let $\ord(a)$ denote the order of $a$ in $(\ZZ/p\ZZ)^{\times}/\{\pm 1\}$. 
    For each $p$ we have $\ord(a)=(p-1)/2>3\cdot 8$.
    We let $n=7$ if $p\neq 131, 241$ and $n=8$ otherwise.
    In all cases we check that the inequality
    \begin{equation*}
        8 < \frac{325}{2^{16}}\cdot\frac{p^2-1}{n}
    \end{equation*}
    holds. 
    Thus \cite[Proposition 7.1]{derickx2021torsion} asserts that $p\not\in S(8)$. 
    The supporting \texttt{Magma} computations 
    can be found in the script \texttt{largeprimes.m}.
\end{proof}

In order to complete the proof of Theorem \ref{thm:S8}, it remains to show that assumption (b) of Lemma \ref{MainStrategy} holds for $p=37$. 

\begin{lem}
    $37\not\in S(8)$.
\end{lem}

\begin{proof}
We follow closely the proofs of Lemmas 8.8 and 8.9 of \cite{derickx2021torsion}.
We work with $\ell=2$. Using Stoll's code, we checked that there are no elliptic curves over $\F_{2^d}$ with a point of order $37$ for $d=1$, $2$, $3$, $4$, $5$, $8$. Thus all places on $X_{\F_2}$ of those degrees $d$ must be cuspidal. The curve $X=X_1(37)$ has $18$ rational cusps, and $18$ irrational cusps; the latter are defined over $\QQ(\zeta_{37})^+$. As $2$ is inert in $\QQ(\zeta_{37})^+$, the irrational cusps yield a single place on $X_{\F_2}$ of degree $18$.  Let $x$ be a degree $8$ point on $X=X_1(37)$, and write $\bar{x}$ for the corresponding divisor on $X_{\F_2}$. There are only three possible cases.

\medskip

\noindent \textbf{Case (I).} $\bar{x}$ is the sum of eight rational cusps. We proved that assumption (a) of Lemma \ref{MainStrategy} holds in Section \ref{sec:assumptiona}. Thus $x$ is the sum of eight rational cusps giving a contradiction.

\medskip

\noindent \textbf{Case (II).} $\bar{x}=\bar{D}+\bar{y}$ where $\bar{D}$ is a degree $7$ place on $X_{\F_2}$, and $\bar{y}$ is the reduction of a rational cusp. We note that $t(\bar{y})$ is principal, and $t(\bar{x})$ is principal. Hence $t(\bar{D})$ must be principal. Using Stoll's code we checked that for all degree $7$ places $\bar{D}$ of $X_{\F_2}$, the divisor $t(\bar{D})$ is not principal, giving a contradiction in this case.

\medskip

\noindent \textbf{Case (III).} $\bar{x}=\bar{D}+\bar{y}+\bar{z}$ where $\bar{D}$ is a place of degree $6$ on $X_{\F_2}$ and $\bar{y}$, $\bar{z}$ are reductions of rational cusps, that may be distinct or equal. Again $t(\bar{x})$, $t(\bar{y})$ and $t(z)$ are principal, therefore $t(\bar{D})$ must be principal. We checked using the same code that there is precisely one degree $6$ place on $X_{\F_2}$ (up to the action of the diamond operators) such that $t(\bar{D})$ is principal. As noted in \cite{derickx2021torsion}, the divisor $\bar{D}$ is the reduction of a degree $6$ point $D$ on $X$. To obtain a contradiction, it is enough to show that $D+y+z$
is the unique rational point on $X^{(8)}$ in the residue disk of $\bar{D}+\bar{y}+\bar{z}$.

Continuing in the footsteps of \cite[Lemma 8.8]{derickx2021torsion} we consider the projection $T_{17} : J \rightarrow A$, where $A$ is a $36$ dimensional abelian variety of rank $0$, rational torsion subgroup of odd order; moreover the eigenvalues of $T_{17}$ acting on the eigenforms coming from $A$ are all odd. To show that $D+y+z$ is the unique rational point in the residue disk of $\bar{D}+\bar{y}+\bar{z}$ it is enough to verify that the relevant \lq Derickx matrix\rq\ (see \cite[Proposition 3.7]{derickxthesis}) has rank $8$. Using a basis for $S_2(\Gamma_1(37))$, which has dimension $40$, Stoll's 
code constructs a canonical embedding for $X \subset \mathbb{P}^{39}$. Thus regular differentials on $X$ may be identified with linear combinations of the coordinates on $\mathbb{P}^{39}$. With this identification, Stoll's code determines the differentials $\omega_1,\dotsc,\omega_{36}$ coming from the rank zero quotient $A$. If the two cusps $y$, $z$ are distinct, then the divisor $\bar{D}+\bar{y}+\bar{z}$ is the sum of eight geometric points, say $\bar{D}+\bar{y}+\bar{z}=\bar{p}_1+\cdots+\bar{p}_8$. In this case, the Derickx matrix has a particularly simple form, $M=(\omega_i(p_j))$, and the formal immersion criterion is satisfied if this matrix has rank $8$ (see \cite[Proposition 3.7]{derickxthesis}). We do not know the degree $6$ place $\bar{D}$ on this particular model, but we checked, for all distinct pairs of rational cusps $y$, $z$, and all degree $6$ places $\bar{D^\prime}$ on $X_{\F_2}$ that the matrix $M$ for $\bar{D^\prime}+\bar{y}+\bar{z}$ has rank $8$ as required.

It remains to consider the case where $y=z$. Note that the action of the diamond operators on the rational cusps is transitive, and one of these rational cusps is the $\infty$ cusp. Thus it is enough to show that the Derickx matrix for $\bar{D^\prime}+2 \infty$ has rank $8$ for all degree $6$ places on $X_{\F_2}$. Write $\bar{D}^\prime=p_1+\cdots+p_6$ where the $p_i$ are geometric points. Then the Derickx matrix is 
\[
M= \begin{pmatrix}
\omega_1(p_1) & \omega_1(p_2) & \cdots & \omega_1(p_6) & a_1(\omega_1) & a_2(\omega_1) \\
\vdots & \vdots & & \vdots & \vdots & \vdots\\
\omega_{36}(p_1) & \omega_{36}(p_2) & \cdots & \omega_{36}(p_6) & a_1(\omega_{36}) & a_2(\omega_{36}) \\
\end{pmatrix};
\]
here $a_1(\omega)$ and $a_2(\omega)$ are the first two coefficients in the expansion of $\omega$ in terms of any uniformizer at $\infty$. Our differentials $\omega_1,\dotsc,\omega_{36}$ come from cusp expansions around $\infty$, with the cusp expansion $f=a_1 q + a_2 q^2+\cdots$ giving the differential $\omega=f(q) dq/q=(a_1+a_2 q+\cdots) dq$. As $q$ is a uniformizer for the $\infty$-cusp we may use these coefficients $a_1$, $a_2$ in the Derickx matrix. We computed all the possible matrices $M$ and checked that they indeed have rank $8$. This completes the proof. 
The supporting computations can be found in the script \texttt{X137.m}.
\end{proof}

\section*{Data availability statement}

All supporting \texttt{Magma} computations 
can be found in the following public GitHub repository: \url{https://github.com/MaleehaKhawaja/deg8torsionprimes}.

\section*{Conflict of interest statement}
The author has no conflicts of interest to declare that are relevant to the content of this article.

\bibliographystyle{abbrv}
\bibliography{torsion}

\end{document}